\pdfoutput=1
\documentclass[12pt]{amsart}
\usepackage[margin=1in,includefoot,footskip=30pt]{geometry}
\usepackage[T1]{fontenc}
\usepackage{txfonts}
\usepackage{amsmath,amsthm}
\usepackage{hyperref}

\newtheorem{theorem}{Theorem}
\newtheorem{lemma}{Lemma}

\theoremstyle{definition}

\newcommand{\RR}{\ensuremath{\mathbb{R}}}
\newcommand{\ZZ}{\ensuremath{\mathbb{Z}}}

\makeatletter
\newcommand{\leqnomode}{\tagsleft@true}
\newcommand{\reqnomode}{\tagsleft@false}
\makeatother

\begin{document}
\title{Dilated Floor Functions That Commute}
\author{Jeffrey C. Lagarias}
\address{Dept.\ of Mathematics, University of Michigan, Ann Arbor, MI 48109--1043}
\email{\href{mailto:lagarias@umich.edu}{lagarias@umich.edu}}
\author{Takumi Murayama}
\address{Dept.\ of Mathematics, University of Michigan, Ann Arbor, MI 48109--1043}
\email{\href{mailto:takumim@umich.edu}{takumim@umich.edu}}
\author{D. Harry Richman}
\address{Dept.\ of Mathematics, University of Michigan, Ann Arbor, MI 48109--1043}
\email{\href{mailto:hrichman@umich.edu}{hrichman@umich.edu}}
\makeatletter
\hypersetup{pdftitle={\@title},pdfauthor={Jeffrey C. Lagarias, Takumi Murayama, and D. Harry Richman}}
\makeatother

\date{November 16, 2016}

\begin{abstract}
  We determine all pairs of real numbers $(\alpha, \beta)$
  such that the dilated floor functions
  $\lfloor \alpha x\rfloor$ and $\lfloor \beta x\rfloor$
  commute under composition, i.e., such that
  $\lfloor \alpha \lfloor \beta x\rfloor\rfloor = \lfloor \beta \lfloor \alpha
  x\rfloor\rfloor$ holds for all real $x$.
\end{abstract}
\maketitle

%%%%%%%
% Section 1: Introduction 
%%%%%%%
\section{Introduction}\label{sec1}
The {\em floor function}
$\lfloor x \rfloor$ rounds a real number down to the nearest integer.
The {\em ceiling function} $\lceil x \rceil$,
which rounds up to the nearest integer,  satisfies
\begin{equation}\label{ceiling-identity}
  \lceil x \rceil = - \lfloor -x \rfloor.
\end{equation}
These two fundamental operations discretize (or quantize) real numbers
in different ways. 
The names {\em floor function} and {\em ceiling function}, along with
their notations, were coined in 1962 by Kenneth E. Iverson \cite[p.\ 12]{Iver62},
in connection with the programming language $APL$.
Graham, Knuth, and Patashnik \cite[Chap.\ 3]{GKP94} note this history and
give many interesting properties of these functions.

We study the floor function applied to a linear function $\ell_{\alpha}(x) =
\alpha x$, yielding the {\em dilated} floor function
$f_{\alpha}(x) = \lfloor \alpha x \rfloor$, where $\alpha$ is a real number. Dilated floor
functions arise in constructing digital straight lines, which are ``lines'' drawn
on two-dimensional graphic displays using pixels, and are discussed further below.
This note addresses the question: {\em When do two dilated floor
functions commute under composition of functions?} Linear functions always commute
under composition and satisfy the identities
\begin{equation}\label{linear-identity}
  \ell_{\alpha} \circ \ell_{\beta}(x) = \ell_{\beta} \circ \ell_{\alpha} (x) =
  \ell_{\alpha \beta}(x) \quad \text{for all}\ x \in \RR.
\end{equation}
However, discretization generally destroys such commutativity. We have the
following.
%%%%%%
%
% Theorem 1
%
%%%%%%
\begin{theorem} \label{thm:1}
  The complete set of all $(\alpha, \beta) \in \RR^2$ such that
  \[
    \lfloor \alpha \lfloor \beta x\rfloor \rfloor
    = \lfloor \beta \lfloor \alpha x\rfloor \rfloor
  \]
  holds for all $x \in \RR$ consists of:
  \begin{enumerate}
    \item[(i)]
      three continuous families $(\alpha, \alpha)$, $(\alpha, 0)$, $(0, \alpha)$ for
      all $\alpha \in \RR$;
    \item[(ii)]
      the infinite discrete family 
      \[
        \left\{ (\alpha, \beta) = \left(\frac{1}{m}, \frac{1}{n}\right) : m, n
        \ge 1 \right\},
      \]
      where $m,n$ are positive integers. (The families overlap when $m=n$.)
  \end{enumerate}
\end{theorem}
 
The  interesting  feature of this  classification  is the 
existence of the infinite discrete family (ii) of solutions
where commutativity survives. 
The family (ii) 
fits together to form an infinite family of
pairwise commuting functions $T_m(x) := f_{1/m}(x) = \lfloor \frac{1}{m}x \rfloor$
for integers $m \ge 1$.
Moreover, these functions  satisfy 
for all $m, n \ge 1$ the further relations
\[
  T_m \circ T_n(x)  = T_n \circ T_m (x)= T_{mn}(x)\quad \text{for all}\ x \in \RR,
\]
which are the same relations satisfied by composition of linear functions
\eqref{linear-identity}. 

One can ask an analogous question for dilated ceiling functions: {\em When do two
dilated ceiling functions commute?} The resulting classification turns out to be identical.
To see this, set $g_{\alpha}(x) := \lceil \alpha x \rceil$. 
Using the identity \eqref{ceiling-identity},
we deduce that for any $\alpha, \beta$,
\begin{equation*}\label{ceiling-commute}
  f_{\alpha} \circ f_{\beta} (x) = - g_{\alpha} \circ g_{\beta} (-x), \quad
  \mbox{for all}\ x \in \RR.
\end{equation*}
Since $x \mapsto -x$ is a 
bijection of the domain $\RR$ to itself, we see that $g_{\alpha}$ and $g_{\beta}$ commute
under composition if and only if $f_{\alpha}$ and $f_{\beta}$ commute
under composition.

The  commuting family (ii) was noted by 
 Cardinal \cite[Lemma 6]{Car10} in a number-theoretic context.
 He studied certain semigroups of integer  matrices,  constructed using
 the floor function, from which he constructed a family of symmetric integer matrices 
that  he  related to the Riemann hypothesis. 
Also from this number-theoretic perspective, symmetry properties of the  solutions 
may be important. Both sets of solutions (i) and (ii) are invariant under exchange $(\alpha, \beta)$
to $(\beta, \alpha)$. However:
\begin{enumerate}
\item[(1)]
The set  of all continuous solution parameters (i) is invariant
  under the reflection symmetry taking $(\alpha, \beta)$ to $(- \alpha, -\beta)$,
 while the discrete solutions (ii) break this symmetry. 
  \item[(2)]
  If one restricts to strictly nonzero parameters, then
  the continuous solution parameters (i) are invariant under the symmetry taking $(\alpha, \beta)$
  to $(\frac{1}{\alpha}, \frac{1}{\beta})$, while the  discrete solutions (ii) break this symmetry.
  \end{enumerate}

In the next section we prove Theorem \ref{thm:1}, and in the final section, we
discuss the problem in the general context of digital straight lines.

%%%%%%%

%***************
% Section 2: Proofs
%*********************
\section{Proof of Theorem \ref{thm:1}} \label{sec2}
Two immediate cases where commutativity holds are $\alpha=0$ or $\beta=0$: In these cases, the
functions $f_\alpha$ and $f_\beta$ commute since their
composition is the zero function. In what follows, we suppose that
$\alpha\beta \ne 0$, and  then we reparameterize the problem in terms of inverse
parameters $(1/\alpha,1/\beta)$, which will simplify the resulting formulas. 

%%%%%%%%%%%%%%%%%%%%%%%%
% The floor function is upper semi-continuous.
%%%%%%%%%%%%%%%%%%%%%%%%%%%%

We prove Theorem \ref{thm:1} by a case analysis that depends on the signs
of $\alpha$ and $\beta$.
The proofs analyze the jump points in the graphs of 
$f_{1/\alpha}\circ f_{1/\beta}(x)$. 
We define for real $y$ the {\it upper level set} at level $y$: 
\[
  S_{1/\alpha,1/ \beta}(y) := \{x : f_{1/\alpha} \circ f_{1/ \beta}(x) \ge y\} =
  (f_{1/\alpha} \circ f_{1/\beta})^{-1}[y,\infty).
\]
The commutativity property asserts the equality
$S_{1/\alpha,1/\beta}(n) = S_{1/\beta,1/\alpha}(n)$ of upper level sets
for all $n \in \ZZ$,
and the converse holds 
because  the range of ${f_{1/\alpha} \circ f_{1/\beta}}$ is a subset of $\ZZ$.
%%%%%%%%%%%%%%%%%%
The key formulas are identities determining these upper level sets given in Lemmas \ref{lem:200}
and \ref{lem:201},
leading to formulas characterizing  commutativity when $\alpha, \beta >0$ and $\alpha, \beta <0$  given in Lemmas \ref{lem:2} 
and \ref{lem:5}, respectively.

%***************
% Sec 2.1
%*********************
\subsection*{Case 1. Both \texorpdfstring{$\alpha$}{alpha} and
\texorpdfstring{$\beta$}{beta} are positive} \label{sec31}
We begin with  a formula for the upper level sets at integer points.
%%%%%%%%%%
% Lemma 2.1
%%%%%%%%%%
\begin{lemma}\label{lem:200}
For $\alpha, \beta >0$ and each $n \in \ZZ$, the upper level set is
$$
S_{1/\alpha, 1/\beta}(n) = \left[  \beta \lceil n\alpha \rceil , \infty \right).
$$
\end{lemma}
\begin{proof}
  We have the following implications:
  \reqnomode
  \begin{alignat*}{3}
    x \in S_{1/\alpha, 1/\beta} (n) & \Leftrightarrow{} \left\lfloor \frac{1}{\alpha} \left\lfloor \frac{1}{\beta} x \right\rfloor \right\rfloor \ge n &\quad& \text{(by definition)}\\
    &\Leftrightarrow{} \frac{1}{\alpha} \left\lfloor \frac{1}{\beta} x\right\rfloor \ge n &\quad& \text{(the right side is in $\ZZ$)}\\
    &\Leftrightarrow{} \left\lfloor \frac{1}{\beta} x \right\rfloor \ge n\alpha &\quad& \text{(since $\alpha > 0$)}\\
    &\Leftrightarrow{} \left\lfloor \frac{1}{\beta} x \right\rfloor \ge \lceil n\alpha \rceil &\quad& \text{(the left side is in $\ZZ$)} \\
    &\Leftrightarrow{} \frac{1}{\beta} x \ge \lceil n\alpha \rceil &\quad& \text{(the right side is in $\ZZ$)}\\
    &\Leftrightarrow{} x \ge \beta \lceil n\alpha \rceil &\quad& \text{(since
    $\beta >0$).}\tag*{\qedhere}
  \end{alignat*}
  \leqnomode
\end{proof}

%%%%%%%%%%
% Lemma 2.2
%%%%%%%%%%
\begin{lemma}\label{lem:2}
  For $\alpha, \beta >0$, the function $f_{1/\alpha}$ commutes with $f_{1/\beta}$ if
  and only if the equality 
  \begin{equation}\label{ceiling-fact}
    \beta \lceil n \alpha \rceil = \alpha \lceil n \beta \rceil
  \end{equation}
  holds for all integers $n \in \ZZ$.
\end{lemma}
\begin{proof}
  By Lemma \ref{lem:200}, we have
  $x \in S_{1/\alpha, 1/\beta}(n)$ if and only if $x \ge \beta \lceil n\alpha \rceil$.
  Similarly,
  $x \in S_{1/\beta, 1/\alpha}(n)$ if and only if $x \ge \alpha \lceil n\beta \rceil$,
  so that commutativity of the functions is equivalent to the desired equality
  of ceiling functions.
\end{proof}

%%%%%%%%%%
% Lemma 2.3
%%%%%%%%%%
\begin{lemma}\label{lem:3}
  For $\alpha,\beta >0$, the function $f_{1/\alpha}$ commutes with $f_{1/\beta}$
  if and only if either $\alpha = \beta$ or if
  $\alpha$ and $\beta$ are both positive integers.
\end{lemma}
\begin{proof}
  If $\alpha = \beta$ then commutativity clearly holds. If $\alpha,\beta$ are
  both (positive) integers, then the relation \eqref{ceiling-fact} holds for all
  $n \in \ZZ$ since the ceiling functions have no effect. Hence, commutativity
  holds.
  
  The remaining case is that where at least one of $\alpha,\beta$ is not an integer; without loss of
  generality, assume  $\alpha$ is not an integer. We write $\lceil \alpha \rceil = A \ge 1$,
  with $A > \alpha$, and $\lceil \beta \rceil = B \ge 1$.
  We show that commutativity occurs only if $\alpha=\beta$.

  Starting from Lemma \ref{lem:2}, the relation \eqref{ceiling-fact} can be
  rewritten 
  \begin{equation}\label{key-reln}
    \frac{\alpha}{\beta} = \frac {\lceil n \alpha\rceil}{\lceil n\beta\rceil},
  \end{equation}
  whenever the term $\lceil n\beta \rceil$ is non-vanishing; here $\lceil n\beta \rceil \ge 1$ holds for $n \ge 1$.

  Since $\alpha < A$, there exists a finite $n \ge 2$ such that $\lceil k\alpha
  \rceil= kA$ for $1 \le k \le n-1$, while $\lceil n\alpha \rceil = nA - 1$.
  Now, \eqref{key-reln} requires
  \[
    \frac{\alpha}{\beta} = \frac{A}{B} = \frac{\lceil k\alpha \rceil}{\lceil k\beta \rceil}
    \quad\text{for all}\ k \ge 1.
  \]
  By induction on $k \ge 1$, this relation implies $\lceil k \beta\rceil = kB$ for
  $1 \le k \le n-1$. It also implies that $\lceil n \beta \rceil = nB$ or $nB-1$. 
  The relation \eqref{key-reln} for $k=n$ becomes
  \[
    \frac{\alpha}{\beta} = \frac{A}{B} = \frac{\lceil n \alpha \rceil}{\lceil
    n\beta \rceil} = \frac{nA - 1}{\lceil n \beta \rceil},
  \]
  which rules out $\lceil n \beta \rceil = nB$. Thus, $\lceil n \beta \rceil =
  nB-1$, and we now have
  \[
    \frac{A}{B} = \frac{nA-1}{nB-1}.
  \]
  Clearing denominators yields $nAB - A = nAB - B$, whence $A=B$. Thus, we have
  $\frac{\alpha}{\beta} = \frac{A}{B} = 1$, so that $\alpha = \beta$
  as asserted.
\end{proof}

%***************
% Sec 2.2
%*********************
\subsection*{Case 2. Both \texorpdfstring{$\alpha$}{alpha} and
\texorpdfstring{$\beta$}{beta} are negative} \label{sec32}

We obtain a criterion which parallels Lemma \ref{lem:2} in the positive case.

%%%%%%%%%%
% Lemma 2.4
%%%%%%%%%%
\begin{lemma}\label{lem:201}
  For $\alpha, \beta <0$ and each $n \in \ZZ$, the upper level set is 
  $$
    S_{1/\alpha, 1/\beta}(n) = \left(  \beta \lfloor n \alpha \rfloor + \beta , \infty \right).
  $$
\end{lemma}
\begin{proof}
  We have the following implications:
  \reqnomode
  \begin{alignat*}{3}
    x \in S_{1/\alpha, 1/\beta} (n) & \Leftrightarrow{} \left\lfloor \frac{1}{\alpha} \left\lfloor \frac{1}{\beta} x \right\rfloor \right\rfloor \ge n &\quad& \text{(by definition)}\\
    &\Leftrightarrow{} \frac{1}{\alpha} \left\lfloor \frac{1}{\beta} x\right\rfloor \ge n &\quad& \text{(the right side is in $\ZZ$)}\\
    &\Leftrightarrow{} \left\lfloor \frac{1}{\beta} x \right\rfloor \le n \alpha &\quad& \text{(since $\alpha < 0$)}\\
    &\Leftrightarrow{} \left\lfloor \frac{1}{\beta} x \right\rfloor \le \lfloor n \alpha \rfloor &\quad& \text{(the left side is in $\ZZ$)} \\
    &\Leftrightarrow{} \frac{1}{\beta} x < \lfloor n \alpha \rfloor + 1 &\quad& \text{(the right side is in $\ZZ$)}\\
    &\Leftrightarrow{} x > \beta \lfloor n \alpha \rfloor + \beta &\quad&
    \text{(since $\beta< 0$).}\tag*{\qedhere}
  \end{alignat*}
  \leqnomode
\end{proof}

%%%%%%%%%%
% Lemma 2.5
%%%%%%%%%%
\begin{lemma}\label{lem:5}
  For $\alpha,\beta < 0$, the function $f_{1/\alpha}(x)$ commutes with
  $f_{1/\beta}(x)$ if and only if the equality 
  \[
    \beta \lfloor n \alpha\rfloor+ \beta = \alpha \lfloor n \beta \rfloor + \alpha
    \]
  holds for all integers $n \in \ZZ$.
\end{lemma}
\begin{proof}
  By Lemma \ref{lem:201}, we have
  $x \in S_{1/\alpha, 1/\beta} (n)$ if and only if $x > \beta \lfloor n \alpha \rfloor + \beta$.
  Similarly, we have 
  $x \in S_{1/\beta, 1/\alpha} (n)$ if and only if $x > \alpha \lfloor n \beta
  \rfloor + \alpha$, so that commutativity of the functions is equivalent to the
  desired equality.
\end{proof}

%%%%%%%%%%
% Lemma 2.6
%%%%%%%%%%

\begin{lemma}\label{lem:6}
  For $\alpha, \beta < 0$, the function $f_{1/\alpha}$ commutes with
  $f_{1/\beta}$ if and only if $\alpha = \beta$.
\end{lemma}
\begin{proof}
  Choose $n=0$ in Lemma \ref{lem:5}.  
  We obtain that $\alpha=\beta$ is a necessary condition for commutativity.
  But this condition is obviously sufficient.
\end{proof}

%***************
% Sec 2.3 Opposite sign
%*********************
\subsection*{Case 3. \texorpdfstring{$\alpha$}{alpha} and
\texorpdfstring{$\beta$}{beta} are of opposite signs} \label{sec33}

%%%%%%%%%%
% Lemma 2.7
%%%%%%%%%%

\begin{lemma}\label{lem:7}
  For $(\alpha, \beta)$ with $\alpha\beta < 0$, the function $f_{1/\alpha}(x)$
  never commutes with $f_{1/\beta}(x)$. 
\end{lemma}
\begin{proof}
  Without loss of generality, we may consider  $\alpha >0$ and $\beta <0$. 
  It suffices to show $S_{1/\alpha,1/\beta}(n) \ne S_{1/\beta,1/\alpha}(n)$.
  We will see that both of these upper level sets start at $- \infty$ and have a finite
  right endpoint.

  We first compute $S_{1/\alpha,1/\beta}(n)$. We can follow the same steps as in
  Lemma \ref{lem:200}, except in the last step where we have instead that $x \in
  S_{1/\alpha,1/\beta}(n)$ if and only if $x \le \beta\lceil n\alpha\rceil$ since
  $\beta < 0$. We obtain for $\alpha>0$ and $\beta <0$ that 
  \[
    S_{1/\alpha, 1/\beta}(n) = (-\infty, \beta  \lceil n\alpha \rceil]
  \]
  is a closed interval.
   
  \par Next, we compute $S_{1/\beta,1/\alpha}(n)$. We can follow the same steps as in
  Lemma \ref{lem:201}, except in the last step where we have instead that $x \in
  S_{1/\beta,1/\alpha}(n)$ if and only if $x< \alpha \lfloor n \beta\rfloor + \alpha$
  since $\alpha > 0$. We find in this case  that
  \[
    S_{1/\alpha, 1/\beta}(n) = (-\infty, \alpha \lfloor  n\beta\rfloor +\alpha)
  \]
  is an open interval.
  It follows that the two functions cannot commute.
\end{proof}
   
The case analysis is complete, and  Theorem \ref{thm:1} follows.

%***************
% Section 3: Extensions Proofs
%*********************

\section{Digital Straight Lines}\label{sec3}
The mathematical study of digital straight lines,
which are ``lines'' drawn on two-dimen\-sional graphic displays represented by
pixels, was initiated by A. Rosenfeld \cite{Ros74} in 1974.
For more recent work, see Klette and Rosenfeld \cite{KR04} and
Kiselman \cite{Kis11}. In drawing a digital image of the line
$\ell_{\alpha,\gamma}(x) := \alpha x + \gamma$, a simple recipe is to associate
to the abscissa $n = \lfloor x \rfloor$ the pixel $(\lfloor x\rfloor, \lfloor
\alpha \lfloor x \rfloor + \gamma \rfloor) \in \ZZ^2$
(more complicated recipes are used in practice). Bruckstein \cite{Bru91} noted
self-similar features of digital straight lines, relating them to the continued
fraction expansion of their slopes; see also McIlroy \cite{McI91}. In contrast,
our proof of  Theorem \ref{thm:1} does not require continued fractions.

From the digital straight line viewpoint, one can view $\lfloor \alpha \lfloor \beta
x\rfloor\rfloor$ as a step function approximation to the straight line 
$\ell_{\alpha \beta}(x) := \alpha \beta x$ in the sense that the difference
function
\[
  h_{\alpha,\beta}(x) := \lfloor \alpha \lfloor \beta x\rfloor\rfloor - \alpha \beta x 
\]
is a bounded function. This difference function is explicitly given by a
combination of iterated fractional part functions
\(
  h_{\alpha,\beta}(x) = -\alpha \{\beta x\} - \{\alpha(\beta x - \{\beta x\})\}
\)
so is a bounded {\em generalized polynomial} in the sense of Bergelson and Leibman
\cite{BerL07}.
The commutativity problem studied
here is that of determining when the generalized polynomial
$h_{\alpha, \beta} (x) - h_{\beta, \alpha}(x)$ is identically zero.

Commutativity questions under composition can be considered for general
digital straight lines 
such as $f_{\alpha, \gamma}(x) := \lfloor \alpha x + \gamma\rfloor$.
However, general linear functions $\ell_{\alpha, \gamma}(x) = \alpha x + \gamma$
with distinct nonzero $\gamma$ do
not commute under composition. We do not know whether any interesting new
commuting pairs occur in this more general context.

\subsection*{Acknowledgement}
The authors thank S. Mori for a helpful comment. The first author received financial support from NSF grant DMS-1401224.

\end{document}